\documentclass[a4paper,draft,12pt]{article}
\pdfoutput=1
\usepackage{amssymb,amsmath}
\usepackage{amsthm}
\usepackage{tikz}
 \newtheorem{thm}{Theorem}[section]
 \newtheorem{cor}[thm]{Corollary}
 \newtheorem{lem}[thm]{Lemma}
 \newtheorem{prop}[thm]{Proposition}
 \theoremstyle{definition}
 \newtheorem{defn}[thm]{Definition}
 \theoremstyle{remark}
 \newtheorem{rem}[thm]{Remark}
 \newtheorem*{ex}{Example}
 \numberwithin{equation}{section}

\def\s{\smallskip}
\def\m{\medskip}

\def\noi{\noindent}

\def\d{\displaystyle}

\def\text{\mbox}

\begin{document}

\author{S.Capparelli and A. Del Fra}
\title{Some results on Dyck paths\\ and Motzkin paths}
\maketitle
\begin{abstract}
We introduce an equivalence relation on the set of Dyck paths and some operations on them. We determine a formula for the cardinality of those equivalence classes and use this information to obtain a combinatorial formula for the number of Dyck and Motzkin paths of a fixed length.
\end{abstract}

\section{Introduction}In  \cite{CD}, among other things, we study a certain equivalence relation on the set of Dyck paths
and we compute the cardinality of the equivalence classes. We then use this information to give a combinatorial formula for the number of Dyck and Motzkin paths of a fixed length.
Some of these results were obtained by \cite{FLA} using continued fractions. In this paper we give a more detailed exposition of of the results in \cite{CD} including some proofs there omitted.
\section{ Motzkin paths with horizontal steps at a single level}
\begin{defn}
 A {\it Motzkin path} of length $n$ is a path on the integral lattice $\mathbb{Z}\times \mathbb{Z}$ starting from $(0,0)$ and ending in $(k,0)$ using $n$ steps according to the vectors  $(1,1)$, $(1,0)$, $(1,-1)$ and never going below the $x$-axis, (\cite{AIG}).
\end{defn}
For example

$
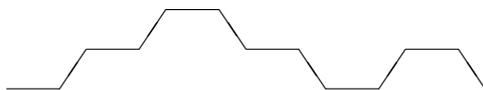
\begin{picture}(300,-70)(-50,30)
\setlength{\unitlength}{.7mm}

\put(2,1){}
\put(13,1.7){}
\put(17,9){}
\put(28,9.7){}
\put(33,17){}
\put(37.5,9.7){}
\put(46.5,9){}
\put(52,1.7){}
\put(62,1){}
\put(73,1.7){}
\put(79,9){}
\put(82,1.7){}
\put(0,0){\line(1,0){10}}
\put(10,0){\line(2,3){5}}
\put(15,7.5){\line(1,0){10}}
\put(25,7.5){\line(2,3){5}}
\put(30,15){\line(1,0){10}}
\put(40,15){\line(2,-3){5}}
\put(45,7.5){\line(1,0){10}}
\put(55,7.5){\line(2,-3){5}}
\put(60,0){\line(1,0){10}}
\put(70,0){\line(2,3){5}}
\put(75,7.5){\line(1,0){10}}
\put(85,7.5){\line(2,-3){5}}

\end{picture}$

\vskip 0.5in

\begin{defn}
 A Motzkin path with no $(1,0)$ steps is called a {\it Dyck path} (\cite{DEU}).
\end{defn}

In this section we want  to count the number of Motzkin paths that have horizontal steps only at a given fixed level. For this, it is important to count  how many Dyck paths have two feet, three feet, four feet, etc. defining a foot to be a point where the path touches the fixed level. For example,  it is clear that there is only one Dyck path with length two and this has two feet at level zero:

 
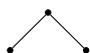
\vspace{.5cm}
\begin{tikzpicture}[scale=0.5]
\tikzstyle{every node}=[draw,circle,fill=black,minimum size=2pt,inner sep=0pt,radius=1pt]
      \node (A) at (0,0) {};
       \node (B) at (1,1)  {} ;
      \node (C) at (2,0) {};
       \draw (A)--(B);
      \draw (B)--(C);
       \end{tikzpicture}

 \noi Among the paths with length 3, one has two feet, one has three feet at level zero.

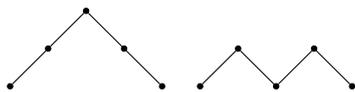
\vspace{.5cm}
\begin{tikzpicture}[scale=0.5]
\tikzstyle{every node}=[draw,circle,fill=black,minimum size=2pt,inner sep=0pt,radius=1pt]
      \node (A) at (0,0) {};
       \node (B) at (1,1)  {} ;
      \node (C) at (2,2) {};
\node (D) at (3,1) {};
\node (E) at (4,0) {};
       \draw (A)--(B);
      \draw (B)--(C);
      \draw (C)--(D);
      \draw (D)--(E);
      \node (A1) at (5,0) {};
       \node (B1) at (6,1)  {} ;
      \node (C1) at (7,0) {};
\node (D1) at (8,1) {};
\node (E1) at (9,0) {};
       \draw (A1)--(B1);
      \draw (B1)--(C1);
      \draw (C1)--(D1);
      \draw (D1)--(E1);
       \end{tikzpicture}

 \noi Among the five paths with  length 6, two have three feet, two have two feet, one has four feet at level zero:

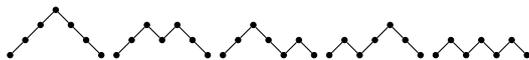
\vspace{.5cm}
\begin{tikzpicture}[scale=0.2]
\tikzstyle{every node}=[draw,circle,fill=black,minimum size=2pt,inner sep=0pt,radius=1pt]
      \node (A) at (0,0) {};
       \node (B) at (1,1)  {} ;
      \node (C) at (2,2) {};
      \node (D) at (3,3) {};
      \node (E) at (4,2) {};
      \node (F) at (5,1) {};
      \node (G) at (6,0) {};
       \draw (A)--(B);
      \draw (B)--(C);
      \draw (C)--(D);
      \draw (D)--(E);
      \draw (E)--(F);
      \draw (F)--(G);
      \node (A1) at (7,0) {};
       \node (B1) at (8,1)  {} ;
      \node (C1) at (9,2) {};
      \node (D1) at (10,1) {};
      \node (E1) at (11,2) {};
      \node (F1) at (12,1) {};
      \node (G1) at (13,0) {};
       \draw (A1)--(B1);
      \draw (B1)--(C1);
      \draw (C1)--(D1);
      \draw (D1)--(E1);
      \draw (E1)--(F1);
      \draw (F1)--(G1);
      \node (A2) at (14,0) {};
       \node (B2) at (15,1)  {} ;
      \node (C2) at (16,2) {};
      \node (D2) at (17,1) {};
      \node (E2) at (18,0) {};
      \node (F2) at (19,1) {};
      \node (G2) at (20,0) {};
       \draw (A2)--(B2);
      \draw (B2)--(C2);
      \draw (C2)--(D2);
      \draw (D2)--(E2);
      \draw (E2)--(F2);
      \draw (F2)--(G2);
             \node (A3) at (21,0) {};
       \node (B3) at (22,1)  {} ;
      \node (C3) at (23,0) {};
      \node (D3) at (24,1) {};
      \node (E3) at (25,2) {};
      \node (F3) at (26,1) {};
      \node (G3) at (27,0) {};
       \draw (A3)--(B3);
      \draw (B3)--(C3);
      \draw (C3)--(D3);
      \draw (D3)--(E3);
      \draw (E3)--(F3);
      \draw (F3)--(G3);
            \node (A4) at (28,0) {};
       \node (B4) at (29,1)  {} ;
      \node (C4) at (30,0) {};
      \node (D4) at (31,1) {};
      \node (E4) at (32,0) {};
      \node (F4) at (33,1) {};
      \node (G4) at (34,0) {};
       \draw (A4)--(B4);
      \draw (B4)--(C4);
      \draw (C4)--(D4);
      \draw (D4)--(E4);
      \draw (E4)--(F4);
      \draw (F4)--(G4);
      \end{tikzpicture}

\noi If we arrange the resulting numbers in a Pascal-like triangle, we have:
\begin{center}
    \begin{tabular}{ | l | l | l | l | l |l |l |}
    \hline
     & 1-ped  & 2-ped & 3-ped &4-ped &5-ped &6-ped \\
    \hline
    0 steps & 1  & 0 & 0 &0& 0&0\\
    \hline
    2 steps & 0  & 1 & 0&0&0&0\\
    \hline
    4 steps &0& 1  & 1 & 0&0&0\\
    \hline
    6 steps &0& 2  & 2 & 1& 0 & 0\\ \hline
    8 steps &0& 5 & 5 & 3&1&0
     \\ \hline
    10 steps &0& 14 & 14 & 9&4&1 \\ \hline
     12 steps &0& 42 & 42 & 28  &14&5\\  \hline
    \end{tabular}
\end{center}
See below for a definition of $n$-ped.
\begin{rem}\label{conjconv}
The natural conjecture one can formulate by looking at this table is that the second column is made up of the Catalan numbers, while the other columns are convolutions of the Catalan number sequence. We shall have something to say about this later on.
\end{rem}

To count Motzkin path with horizontal level at a single level , we need to count the number of Dyck paths that at a certain level have a given number of feet. This can be done in a recursive fashion. Start by counting the number of feet of Dyck paths at level 0. For this we use a recursive construction suggested by the following recursive formula of the Catalan numbers:
 \begin{equation}\label{leggeCatalan}
C(n)=\sum_{k=0}^{n-1}C(n-1-k)C(k).
\end{equation}
We interpret this formula as follows.
Any length $2n$  Dyck path can be obtained by a combination of two operations. The first operation consists in adding to each path of length  $2(n-1-k)$ an up step at the beginning and a down step at the end. We  call this operation {\bf lifting}. The second operation consists in {\bf gluing} at the end of it any path of length  $2k$,  $k=0,\ldots, n-1$.

Next, we observe that every  Dyck path is obtained by a sequence of up and down steps, indicated by   U and D, respectively, in such a way that the total number of  U's is equal the number of  D's and that at each  step the number of  U's be not less than the number of D's. Let's associate to each   U the number  1 and to each D the number $-1$. Starting from  0, let's sum at each step the numbers  1 and $-1$ up to that point. For a length $2n$ path, the {\it associated } sequence thus obtained starts and ends with   0 and the maximum number appearing is at most $n$. For each integer  $i\ge 0$, denoted by  $p(i)$ the number of times that  $i$ appears in the sequence, we shall say that the path  is $p(i)$-ped at the level $i$. As an example, the length 14 path UUDUUDDUUDUDDD has the associated sequence 012123212323210.  In other words, the  sequence 012123212323210 is simply the sequence of the $y$ coordinate of the points touched by the path. The path is therefore  biped at level 0, quadruped at level 1, 6-ped at level 2, 3-ped at level 3, 0-ped at higher levels.
We may sometimes identify a path with its associated sequence.

For every triple of integers $2n\!\ge\! 0$, $i\ge 0$, $j\ge 0$, we denote by $p^{2n}_{i,j}$ ($
p^{length}_{level,\# feet}
$)
the number of Dyck paths of length $2n$ that are $j$-ped at level $i$.

Our next objective is to compute $p^{2n}_{i,j}$. We shall do this using a recursive rule first on  $i$  and then on  $n$.

The starting point is to compute $p^{2n}_{0,j}$, namely, how many paths of length $2n$ have $j$ feet at level 0, as $n$ grows.
\begin{prop}\label{piediazero}
  We have the following recursive formulas
\begin{equation}p^0_{0,0}=0, p^0_{0,1}=1, p^0_{0,j}=0, j>1\end{equation}

\begin{equation}p^2_{0,0}=0, p^2_{0,1}=0, p^2_{0,2}=1, p^2_{0,j}=0, j>2\end{equation}

  \begin{equation}\label{zeropodi}
p^{2n}_{0,0}=0, p^{2n}_{0,1}=0, \d p^{2n}_{0,2}=\sum_{k=0}^\infty p^{2(n-1)}_{0,k}, n>1
\end{equation}

  \begin{equation}\label{generalpode}
p^{2n}_{0,j}=\sum_{i=0}^{n-2} p^{2(i+1)}_{0,2}p^{2(n-i-1)}_{0,j-1}, j>2, n>1
\end{equation}

\end{prop}
\begin{proof}

There is a unique Dyck  path of length  0, which we call {\it null path} and denote by $\mathbf 0$. The corresponding numerical sequence is 0 and so this path is  1-ped at level 0. The null path clearly has a single foot and nothing else:

$$p^0_{0,0}=0, p^0_{0,1}=1, p^0_{0,j}=0, j>1.$$

The path UD,
\begin{tikzpicture}[scale=0.2]
\tikzstyle{every node}=[draw,circle,fill=black,minimum size=2pt,inner sep=0pt,radius=1pt]
      \node (A) at (0,0) {};
       \node (B) at (1,1)  {} ;
      \node (C) at (2,0) {};
       \draw (A)--(B);
      \draw (B)--(C);
       \end{tikzpicture},
 is the only one of length 2. The associated sequence is 010. It is therefore 2-ped at level 0, 1-ped at level 1 and 0-ped at higher levels. It follows:

$$p^2_{0,0}=0, p^2_{0,1}=0, p^2_{0,2}=1, p^2_{0,j}=0, j>2.$$

Before proceeding further, we observe that by lifting any path we get a 2-ped at level 0 and gluing a 2-ped at level zero to a  $j$-ped at level 0 we get  a $(j+1)$-ped at level 0. Moreover, we notice that there is no path of length greater than zero that is  0-ped or 1-ped at level 0. Hence:

$$p^{2n}_{0,0}=0, p^{2n}_{0,1}=0,  n>0.$$

Suppose we  have determined all $p^{2m}_{0,j}$ up to an even fixed length $2(n-1)$,  $m\le n-1$, let us compute $p^{2n}_{0,j}$. All the  2-ped at level 0 of length  $2n$ are obtained by lifting paths of length $2(n-1)$ and gluing to them the null path $\mathbf 0$.
Therefore
\begin{equation}\label{zeropodi1}
\d p^{2n}_{0,2}=\sum_{k=0}^\infty p^{2(n-1)}_{0,k}.
\end{equation}
\noi
The sum is actually finite since $p^{2(n-1)}_{0,k}$ is zero for $k> n$.

For $j>2$, the $j$-ped paths at level 0 of length $2n$ are constructed by lifting paths of length $2i$, $0\le i\le n-2$, with any number  $k$ of feet, thus becoming  2-ped, and gluing to them $(j-1)$-ped paths of length  $2(n-i-1)$. It follows that

\begin{equation}\label{generalpodetemp}
p^{2n}_{0,j}=\sum_{i=0}^{n-2}(\sum_{k=0}^{\infty}p^{2i}_{0,k})p^{2(n-i-1)}_{0,j-1}.
\end{equation}

Since $\d \sum_{k=0}^{\infty}p^{2i}_{0,k}=p^{2(i+1)}_{0,2}$, by (\ref{zeropodi}),   formula \ref{generalpodetemp} can also be written as

\begin{equation}\label{generalpode1}
p^{2n}_{0,j}=\sum_{i=0}^{n-2} p^{2(i+1)}_{0,2}p^{2(n-i-1)}_{0,j-1}
\end{equation}

This is the general formula for the number of $j$-ped paths   at level zero for any length.
\end{proof}

\begin{rem}

This formula proves the observation made earlier  (see Remark \ref{conjconv}), and it is the interpretation of formula (\ref{leggeCatalan}).
\end{rem}

\begin{prop}\label{piediinalto}
  We have the following recursive formula
\begin{equation}
p^0_{s,0}=1,\, p^0_{s,j}=0, s>0, j>0
\end{equation}

\begin{equation}
 p^{2n}_{s,j}=\sum_{i=0}^{n-1} \sum_{k=0}^{j}p^{2i}_{s-1,k}p^{2(n-i-1)}_{s,j-k}.
 \end{equation}

\end{prop}
\begin{proof}
Let us now proceed to compute the number  $p^{2n}_{1,j}$, for level 1. Obviously

$$p^0_{1,0}=1,\, p^0_{1,j}=0, j>0,$$

$$p^0_{s,0}=1,\, p^0_{s,j}=0, s>1, j>0$$

Suppose we determined all $p^{2m}_{s,j}$,  $m\le n-1$, any $s$ and any $j$, up to a fixed even length $2(n-1)$. We wish to compute $p^{2n}_{s,j}$. The $j$-ped paths  at level $s$ of length  $2n$ are obtained by lifting $k$-ped paths at level $s-1$, $0\le k\le j$, of length $2i$, $0\le i\le n-1$, and gluing to them the $(j-k)$-ped at level $s$ and length $2(n-i-1)$.
Hence
\begin{equation}
 p^{2n}_{s,j}=\sum_{i=0}^{n-1} \sum_{k=0}^{j}p^{2i}_{s-1,k}p^{2(n-i-1)}_{s,j-k}.
 \end{equation}
\end{proof}

The knowledge of the number of $j$-ped Dyck paths at a level $s$ allows us to compute a sort of partial binomial transform at a single level, namely counting  the Motzkin paths having horizontal steps only at a prefixed level.

We now use  Propositions \ref{piediazero} e \ref{piediinalto} to compute the number of  Motzkin paths having horizontal steps only at a level  $k$,  $k\ge 0$ a fixed integer. We shall call such paths $k$-Motzkin paths and shall denote by $m_k^n$ the number of  $k$-Motzkin paths of length $n$.

\begin{thm}\label{thkmotzkin}
  The number of  $k$-Motzkin paths of length $n$ is
  \begin{equation}\label{kmotzkin}
m_k^{(n)}= \sum_{j=0}^{\nu}\sum_{i=0}^{\infty} p^{2j}_{k,i} \binom{n-2j+i-1}{n-2j}.
\end{equation}

\end{thm}
\begin{proof}

Notice that a  $k$-Motzkin path of length  $n$ can be obtained from a  Dyck path $D$ with length $2j$, $0\le 2j\le n$, by adding  $n-2j$ horizontal steps at the  feet of level $k$ of $D$. If $D$ is $i$-ped at level $k$, this can be done in as many ways as there are possibility to express  $n-2j$ as sum of  $i$ integers between  0 and $n-2j$, namely, in $\binom{n-2j+i-1}{n-2j}$ ways. Set $\nu={\left[ {\frac{n}{2}}\right]}$.

It follows
\begin{equation*}
m_{k}^{(n)}= \sum_{j=0}^{\nu}\sum_{i=0}^{\infty} p^{2j}_{k,i} {\binom{n-2j+i-1}{n-2j}}.
\end{equation*}
Notice also that for $n=2j$,  the binomial coefficient is 1  when $i=0$. This is in accordance with the fact   0-ped Dyck paths at level  $k$ of length  $n$ are to be considered in the evaluation of $m_k^n$. Finally, notice that the infinite sum is actually finite since   $p^{2j}_{0,i}$ is zero for $i> j+1$ and $p^{2j}_{k,i}$, with $k>0$, is zero for $i> j$.
\end{proof}

If we give $r$ colors to the horizontal steps at level  $k$, then  formula (\ref{kmotzkin}), where we continue to use a self-explanatory notation, obviously becomes
\begin{prop}
$$m_{k,r}^{(n)}= \sum_{j=0}^{\nu}\sum_{i=0}^{\infty} p^{2j}_{k,i} {\binom{n-2j+i-1}{n-2j}}r^{n-2j}.$$
\end{prop}

\section{Frames of Dyck paths}
Next, we want to add horizontal steps at more then one level. To do this it is not  sufficient anymore to know the number of Dyck paths that are $j$-ped at the various levels, but we need to have  more refined information, that is, we need to know the number of Dyck paths that have a fixed ``frame'' $(i_0,i_1,\ldots)$, according to the definition we are about to give.

Given a  Dyck path $D$, with $i_k$ feet at level $k$, $k=0,1,\ldots$, we call the sequence  $(i_0,i_1,\ldots)$, which is zero from a certain index on, the  {\it frame} of $D$. In other words, the sequence tells us how many feet there are at each level.

\noi This sequence is zero from a certain index on, since, if the length of $D$ is $2n$, one has $i_{n+j}=0$, for $j\ge 1$, namely, the maximum reachable level for a Dyck path of length $2n$ is $n$.

Given an eventually zero sequence $\mathbf{I} =(i_0,i_1, ...)$,  we shall call the integer
  $\d -1+\sum_{k=0}^{\infty}i_k $
  the {\it length} of $\mathbf{I}$ .

We shall say that an eventually zero sequence is  {\it admissible} if it is the frame of some  Dyck path.  For an admissible frame $\mathbf{I}$, its length is the same as the length of a path with that frame and is necessarily even.

For example, frames of Dyck paths of length 0,2,4 are, respectively:
\begin{itemize}

\item $(1,0,0,\ldots) $,

\item $(2,1,0,\ldots)$,

\item $(2,2,1,0,\ldots), (3,2,0,\ldots)$.
\end{itemize}

Notice that two distinct paths with the same length may have the same frame.\! For example, the two paths with length  6:\! $UDUUDD$ and $UUDDUD$, which give rise to the sequences $0101210$ and $0121010$ are, respectively,

\vspace{.3cm}
\begin{tikzpicture}[scale=0.2]
\tikzstyle{every node}=[draw,circle,fill=black,minimum size=2pt,inner sep=0pt,radius=1pt]
      \node (A) at (0,0) {};
       \node (B) at (1,1)  {} ;
      \node (C) at (2,0) {};
      \node (D) at (3,1) {};
      \node (E) at (4,2) {};
      \node (F) at (5,1) {};
      \node (G) at (6,0) {};
       \draw (A)--(B);
      \draw (B)--(C);
      \draw (C)--(D);
      \draw (D)--(E);
      \draw (E)--(F);
      \draw (F)--(G);
      \node (A1) at (10,0) {};
       \node (B1) at (11,1)  {} ;
      \node (C1) at (12,2) {};
      \node (D1) at (13,1) {};
      \node (E1) at (14,0) {};
      \node (F1) at (15,1) {};
      \node (G1) at (16,0) {};
       \draw (A1)--(B1);
      \draw (B1)--(C1);
      \draw (C1)--(D1);
      \draw (D1)--(E1);
      \draw (E1)--(F1);
      \draw (F1)--(G1);
       \end{tikzpicture}

\noi and have the same frame $(3,3,1,0,0,\ldots)$.

\vspace{.3cm}

Obviously ``having the same frame'' is an equivalence relation on the set of   Dyck paths. Notice that lifting two equivalent paths  we still get two equivalent  paths. This is true also if we glue together two paths: if $\mathbf p_i$ is equivalent to $\mathbf q_i$, $i=1,2$ then gluing  $\mathbf p_1$ and $\mathbf p_2$ gives a path that is equivalent to gluing $\mathbf q_1$ and $\mathbf q_2$.
This observation allows us to speak of  {\it lifting a frame} and  {\it gluing two frames}.

It is easy to construct recursively the frames associated to Dyck paths of various lengths, using the same recursive law for the paths.

\noi Lifting the frame $(i_0,i_1,i_2,\ldots)$ one gets the frame $(2,i_0,i_1,i_2,\ldots)$.
While gluing two frames  $(j_0,j_1,j_2,\ldots)$ and $(i_0,i_1,i_2,\ldots)$, we get a frame $(i_0+j_0-1,i_1+j_1,i_2+j_2,\ldots)$.

For the lifting and gluing operations we shall use the following symbols:
$$s(i_0,i_1,i_2,\ldots) = (2,i_0,i_1,i_2,\ldots),$$ $$(i_0,i_1,i_2,\ldots)\wedge (j_0,j_1,j_2,\ldots) = (i_0+j_0-1,i_1+j_1,i_2+j_2,\ldots)$$

It is not difficult to check, for example, that the frames with length 4 can be obtained in this fashion from those with shorter lengths.

The gluing operation on paths is associative but not commutative: in general, $\mathbf u\wedge \mathbf v \ne \mathbf v\wedge \mathbf u$. The same operation on frames however is commutative, namely,  $\mathbf u\wedge \mathbf v$ and $\mathbf v\wedge \mathbf u$ have the same frame.

Actually, any frame can be obtained by the combination of two elementary operations: {\it lifting} of frame $(i_0,i_1,i_2,\ldots)$, which turns  $(i_0,i_1,i_2,\ldots)$ into $s(i_0,i_1,i_2,\ldots)=(2,i_0,i_1,i_2,\ldots)$ and the gluing $(i_0,i_1,i_2,\ldots)$ to the frame of length  2, which we call {\it extension}, that turns  $(i_0,i_1,i_2,\ldots)$ into $a(i_0,i_1,i_2,\ldots)=(i_0,i_1,i_2,\ldots)\wedge (2,1,0,\ldots)=(i_0+1,i_1+1,i_2,\ldots)$. We have therefore:

\m\noi
\begin{thm}\label{costruzione}
Every frame can be obtained by a combination of a lifting and a suitable number of extensions.

\end{thm}

\noi
\begin{proof}
 Consider any path U...D that is not a lifting, i.e., with a frame starting with  $i_0\ge 3$. It follows that its sequence is of the form $01\ldots 101\ldots 10$, with at least one subsequence  $101$ inside. The first triple  $101$ inside comes from an ordered pair  DU. If we eliminate such a pair inside and add instead the pair  $UD$ at the end of the frame, we get a new path with the same frame as the previous one. Iterating such a procedure we get a Dyck path with the same frame as the starting path obtained by gluing a lifting of a suitable path and a finite number, possibly zero, of length 2 paths.
\end{proof}
\m\noi
\begin{thm}
  The number of frames of length  $2n$, with $n>0$, is $2^{n-1}$.
\end{thm}

\noi
\begin{proof}
The frames of length $2n$ are obtained from those of length  $2(n-1)$ by constructing for each of them, say $\mathbf{I}$, the lifting $s(\mathbf{I})$ or the extension $a(\mathbf{I})$. The frame $s(\mathbf{I})$ is different from  $a(\mathbf{I})$ if $n>1$. It follows that the number of frames with length  $2n$ is twice the number of the the frames of length $2(n-1)$.
\end{proof}
\m
Natural generalizations of the extension $a$ and lifting  $s$ operations  can be defined on any sequence of integers eventually zero. The operation $a$ is a bijection on the set of such sequences. Its inverse, denoted by $b$, is defined by:
 $$
b(i_0,i_1,i_2,\ldots)=(i_0-1,i_1-1,i_2,\ldots).
 $$
The operation $s$ is injective and can be inverted on its image, via an operation $r$ defined by:
$$
 r(2,i_{1},i_{2},\ldots)=(i_{1},i_{2},\ldots).
 $$

It is easy to see if a given sequence, which is eventually zero,  is admissible. Indeed, we can, starting from it, trace back the elementary steps that generated it from the basic frame  $(1,0,0,\ldots)$.
The rules can be summed up as follows: every time there is a 2 we erase it by using  $r$, otherwise we subtract 1 from the first two elements, by applying $b$.

\begin{ex}\label{esempio1}
Consider the sequence $(3,6,6,3,1,0,\ldots)$. We wish to see if it is admissible for a path of length  $18=3+6+6+3+1-1$. Tracing backwards we have : $(2,5,6,3,1,0,\ldots)$, $(5,6,3,1,0,\ldots)$, $(4,5,3,1,0,\ldots)$, $(3,4,3,1,0,\ldots)$, $(2,3,3,1,0,\ldots)$, $(3,3,1,0,\ldots)$, $(2,2,1,0,\ldots)$, $(2,1,0,\ldots)$, and  $(1,0,\ldots)$. So the given sequence is admissible.
\end{ex}

\begin{ex}\label{esempio2}
If we take  $(4,5,2,3,1,0,\ldots)$, and we  trace backwards, we get: $(3,\!4,\!2,\!3,\!1,\!0,\ldots)$, $(2,\!3,\!2,\!3,\!1,\!0,\ldots)$, $(3,\!2,\!3,\!1,\!0,\ldots)$, $(2,\!1,\!3,\!1,\!0,\ldots)$, $(1,\!3,\!1,\!0,\ldots)$, and finally $(0,\!2,\!1,\!0,\ldots)$, which is clearly not admissible since there are no Dyck paths of length 2 without feet at the zero level.
\end{ex}
In the preceding arguments, we have implicitly used the following two lemmas whose proof is straightforward:

\begin{lem}\label{one}
  Given an eventually zero integer sequence  $\mathbf{I}$, with $i_0=2$,  $r(\mathbf{I})$ is  admissible if and only if  $\mathbf{I}$ is.
\end{lem}

\begin{lem}\label{two}

    Given an eventually zero integer sequence  $\mathbf{I}$, with  $i_0\ne2$, $b(\mathbf{I})$  is  admissible if and only if  $\mathbf{I}$ is.
\end{lem}

Given an eventually zero integer sequence $(i_0,i_1,i_2,\ldots)$, denote by $i_f$ the last nonzero element and call  $f$ the {\it degree} of the frame.

\m\noi
\begin{thm}\label{CNES}
 An eventually zero sequence of nonnegative integers  $$\mathbf{I}=(i_0,i_1,...,i_f,0,...),$$  with length  $2n$, is admissible if and only if the following conditions are satisfied:
  \begin{itemize}

    \item $i_0=1$, if $f=0$ and $i_0\geq 2$ if $f>0$;
    \item $i_1-i_0\geq 0$ provided $f>1$;
    \item $i_2-i_1+i_0\geq 2$;
    \item $i_3-i_2+i_1-i_0\geq 0$;
    \item $i_4-i_3+i_2-i_1+i_0\geq 2$;
    \item $\cdots\cdots$
    \item  $i_f-i_{f-1}+i_{f-2}+\cdots = -1$ when $f$ is odd;
    \item  $i_f-i_{f-1}+i_{f-2}+\cdots = 1$ when $f$ is even.
  \end{itemize}
\end{thm}
\s\noi
\begin{proof}   The last two conditions can be summarized in the following:
$i_0-i_1+i_2-...+(-1)^fi_f=1$. Notice also that if $f=0$, then the only required condition is  $i_0=1$, while, in case $f=1$, the conditions are  $i_0\geq 2$ and $i_0-i_1=1$.

If $f=0$, $\mathbf{I}=(i_0,0,...)$. Such a sequence is  admissible if and only if it is the frame of the null path, namely, if and only if   $i_0=1$.

Suppose  $f=1$, and use induction on the half-length  $n$. When $n=1$ the only admissible sequence is  $(2,1,0...)$, which satisfies the conditions $i_0\geq 2$ and $i_0-i_1=1$. Assuming the statement true up to the half-length $(n-1)$, we are going to show it for the half-length  $n$. Let $\mathbf{I} =(i_0,i_1,0,...)$ with $i_0+i_1=2n+1$. If $i_0=2$, consider $r(\mathbf{I})=(i_1,0,...)$. This is  admissible if and only if $i_1=1$, and therefore, by Lemma \ref{one}, the result follows. If instead  $i_0\ne 2$, consider $b(\mathbf{I})=(i_0-1,i_1-1,0,...) = (i'_0,i'_1,0,...)$. This, being of length $2(n-1)$, by the induction hypothesis it is admissible if and only if  $i'_0\ge 2$ and $i'_0-i'_1=1$, which are equivalent to $i_0\ge 3$ and $i_0-i_1=1$, hence, again,  Lemma \ref{two} implies the result.

Thus, for degrees  0 and 1 the theorem is proved. Assume  $f\ge 2$. The result is obviously true for length $2n=0$. Assume the result is true up to length $2(n-1)$; we shall prove it for  length $2n$.  We distinguish two cases.

Case A: $i_0=2$. Let $\mathbf{I}=(2,i_1,i_2,...,i_f,0,...)$, with $\d\sum_{k=0}^{\infty}i_k=2n+1$. By applying  $r$, we obtain  $r(\mathbf{I})=(i_{1},i_{2},...,i_f,0,...)= (i_0',i_1',...,i'_{f'},0,...)$. We have:
\begin{equation}
  \begin{split}
    i_0'&=i_1;\\
i'_1-i'_0 &= i_2-i_1=  i_2-i_1+i_0-2;\\
i'_2-i'_1+i'_0 &= i_3-i_2+i_1=i_3-i_2+i_1-i_0+2;\\
i'_3-i'_2+i'_1-i'_0 & = \!i_4\!-\!i_3\!+\!i_2\!-\!i_1\! =  i_4-i_3+i_2-i_1+i_0-2;\\
\cdots\\
i'_0-i'_1+i'_2-...+(-1)^{f'} i'_{f'}&=i_1-i_2+...+(-1)^{f -1}i_f\\
& = i_1-i_2+...+(-1)^{f -1}i_f -i_0+2 \\
&=-(i_0-i_1+i_2-...+(-1)^fi_f)+2.
  \end{split}
\end{equation}
The conditions
\begin{equation}
  \begin{split}
i_0'& \ge 2\\
i'_1-i'_0& \ge0\\
i'_2-i'_1+i'_0&\ge2 \\
i'_3-i'_2+i'_1-i'_0 &\ge0\\
\cdots\\
i'_0-i'_1+i'_2-...+(-1)^{f'} i'_{f'}&=1
 \end{split}
\end{equation}
together with  $i_0=2$ are equivalent to
\begin{equation}
  \begin{split}
i_0&=2\\
i_1-i_0&\ge0\\
i_2-i_1+i_0&\ge2\\
i_3-i_2+i_1-i_0&\ge 0\\
 i_4-i_3+i_2-i_1+i_0&\ge 2\\
\cdots \\
i_0-i_1+i_2-...+(-1)^fi_f&=2\!-\!(i'_0\!-\!i'_1\!+\!i'_2\!-...+\!(-1)^{f'}\! i'_{f'})=\!2\!-\!1=1
\end{split}
\end{equation}
Since the result is true, by the inductive hypothesis, for  $r(\mathbf{I})$, we conclude using  Lemma \ref{one}.
Case B: $i_0\ne 2$.
Let $\mathbf{I}=(i_0,i_1,i_2,...,i_f,0,...)$, with $\d\sum_{k=0}^{\infty}i_k=2n+1$ and $i_0\ne 2$. By applying the operation $b$, we obtain $b(\mathbf{I})=(i_0-1,i_{1}-1,i_{2},...,i_f,0,...)= (i_0',i_1',...,i'_{f'},0,...)$. We have:
\begin{equation}
  \begin{split}
i_0'&=i_0-1\\
i'_1-i'_0 & = i_1-1-i_0+1=  i_1-i_0\\
i'_2-i'_1+i'_0 &= i_2-i_1+1+i_0-1=i_2-i_1+i_0\\
i'_3-i'_2+i'_1-i'_0 &=i_3\!-\!i_2\!+\!i_1\!-\!1\!-\!i_0\!+\!1 =i_3-i_2+i_1-i_0\\
\cdots \\
i'_0-i'_1+i'_2-...+(-1)^{f'} i'_{f'}&=i_0-1-i_1+1+i_2-...+(-1)^fi_f\\
& = i_0-i_1+i_2-...+(-1)^fi_f.
\end{split}
\end{equation}
\noi
Again, in this case the conditions
\begin{equation}
  \begin{split}
i_0'&\ge 2 \\
i'_1-i'_0&\ge0\\
i'_2-i'_1+i'_0&\ge2\\
i'_3-i'_2+i'_1-i'_0 &\ge0\\
\cdots \\
i'_0-i'_1+i'_2-...+(-1)^{f'} i'_{f'}&=1
\end{split}
\end{equation}
\s\noi
are equivalent to
\begin{equation}
  \begin{split}
i_0&>2\\
i_1-i_0& \ge0\\
i_2-i_1+i_0&\ge2\\
i_3-i_2+i_1-i_0&\ge0\\
i_4-i_3+i_2-i_1+i_0&\ge2\\
\cdots\\
i_0-i_1+i_2-...+(-1)^fi_f&=1
\end{split}
\end{equation}
\s
Because of induction, the result is true for $b(\mathbf{I})$, and  Lemma \ref{two}
allows us to conclude.
\end{proof}
\m
The following are immediate consequences of the previous theorem.

\m\noi
\begin{thm}\label{immediate}
 If a frame  $\mathbf{I}=(i_0,i_1,i_2,\ldots,i_f,\ldots)$ with length $2n$ is admissible, then:

\begin{enumerate}
\item
$i_{f-1}> i_f$;
\item
 $i_0= i_1+1\Longleftrightarrow i_1=i_f$;
\item
$2\le i_j\le i_{j-1}+i_{j+1}-2$\quad $(0<j<f-1)$

\end{enumerate}
\end{thm}

\section{Cardinality of a frame}

Recall the two operations defined on the set of Dyck paths:
\begin{itemize}
\item the lifting of a path  $\mathbf u$ denoted by $s(\mathbf u)$;
\item the gluing of two paths $\mathbf u$, $\mathbf v$, denote by $\mathbf u\wedge \mathbf v$.
\end{itemize}

With these two operations one may construct any Dyck path starting from shorter Dyck paths.

Therefore, every Dyck path may be expressed as
$$s^{j_1}(\mathbf x_1)\wedge s^{j_2}(\mathbf x_2)\wedge\ldots \wedge s^{j_t}(\mathbf x_t)$$
where each  $\mathbf x_i$ is again expressible in the same way, with the condition that after a finite number of steps one arrives at expressions of the form $$s^{k_1}(\mathbf 0)\wedge s^{k_2}(\mathbf 0)\wedge\ldots \wedge s^{k_r}(\mathbf 0).$$

For example, let's consider the path  $$UUDUUDDDUDUUUDUDDDUUDD$$ with length 22:

\begin{tikzpicture}[scale=0.4]
\tikzstyle{every node}=[draw,circle,fill=black,minimum size=2pt,inner sep=0pt,radius=1pt]
      \node (A) at (0,0) {};
       \node (B) at (1,1)  {} ;
      \node (C) at (2,2) {};
      \node (D) at (3,1) {};
      \node (E) at (4,2) {};
      \node (F) at (5,3) {};
      \node (G) at (6,2) {};
      \node (H) at (7,1) {};
      \node (I) at (8,0) {};
      \node (J) at (9,1) {};
      \node (K) at (10,0) {};
      \node (L) at (11,1) {};
      \node (M) at (12,2) {};
      \node (N) at (13,3) {};
      \node (O) at (14,2) {};
      \node (P) at (15,3) {};
      \node (Q) at (16,2) {};
      \node (R) at (17,1) {};
      \node (S) at (18,0) {};
      \node (T) at (19,1) {};
      \node (U) at (20,2) {};
      \node (V) at (21,1) {};
      \node (W) at (22,0) {};
       \draw (A)--(B);
      \draw (B)--(C);
      \draw (C)--(D);
      \draw (D)--(E);
      \draw (E)--(F);
      \draw (F)--(G);
      \draw (G)--(H);
      \draw (H)--(I);
      \draw (I)--(J);
      \draw (J)--(K);
      \draw (K)--(L);
      \draw (L)--(M);
      \draw (M)--(N);
      \draw (N)--(O);
      \draw (O)--(P);
      \draw (P)--(Q);
      \draw (Q)--(R);
      \draw (R)--(S);
      \draw (S)--(T);
      \draw (T)--(U);
      \draw (U)--(V);
      \draw (V)--(W);
       \end{tikzpicture}
\vspace{.3cm}

The corresponding sequence is $$01212321010123232101210$$ and it may be expressed as
$$s(\mathbf x_1)\wedge s(\mathbf x_2)\wedge s(\mathbf x_3)\wedge s(\mathbf x_4),$$
where $s(\mathbf x_1)=012123210$, $s(\mathbf x_2)=010$, $s(\mathbf x_3)=012323210$, $s(\mathbf x_4)=01210$, with $\mathbf x_1=0101210$, $\mathbf x_2=\mathbf 0$, $\mathbf x_3=0121210$, $\mathbf x_4=010$.
One  has therefore:

\noi
$\mathbf x_1=s(\mathbf y_1)\wedge s(\mathbf y_2)$, $\mathbf x_3=s(\mathbf y_3)$, $\mathbf x_4 =s(\mathbf 0)$,
with $\mathbf y_1=\mathbf 0$, $\mathbf y_2=010$, $\mathbf y_3=01010$ and so $\mathbf y_2=s(\mathbf 0)$, $\mathbf y_3=s(\mathbf z_1)\wedge s(\mathbf z_2)$,
with
$\mathbf z_1=\mathbf 0$, $\mathbf z_2=\mathbf 0$.

Finally, the assigned path can be expressed as
$$s(s(\mathbf 0)\wedge s^2(\mathbf 0))\wedge s(\mathbf 0)\wedge s^2(s(\mathbf 0)\wedge s(\mathbf 0))\wedge s^2(\mathbf 0).$$

Since a frame is an eventually zero integer sequence, we may think of it as a polynomial. The constant polynomial  1 corresponds to the null frame. Thus  $s(1)$ is the frame of the unique path with length 2, while  $s^2(1)$ and $s(1)\wedge s(1)$ are the frames of the length 4 paths, and so on. It is easy to check that, denoting by  $p(x)$ a frame, one has:
\begin{prop}\label{polinomi}
$s(p(x)) = 2+xp(x), \quad p(x)\wedge q(x) = p(x) + q(x) -1$.
\end{prop}
The procedure to determine the admissibility of a sequence, see Examples \ref{esempio1}, \ref{esempio2}, can be used to determine a ``canonical'' representative of a frame.
\begin{ex}
Consider the frame $(3,4,3,1,0,\ldots)$ and apply to it  the functions $r$ and $b$. We obtain the sequence of frames  $(2,3,3,1,0,...)$, $(3,3,1,0,...)$, $(2,2,1,0,...)$, $(2,1,0,...)$, $(1,0,...)$, thus getting to the null frame. There is, of course, only one path corresponding to the null frame: the null path.

We may now trace this procedure backwards with the operations $s$ and $a$ on the paths, eventually getting a desired canonical representative of the  frame $(3,4,3,1,0,\ldots)$.

 We start form the null path
 \begin{tikzpicture}[scale=0.4]
\tikzstyle{every node}=[draw,circle,fill=black,minimum size=2pt,inner sep=0pt,radius=1pt]
      \node (A) at (0,0) {};
       \end{tikzpicture}
\vspace{.3cm}

this lifted gives:

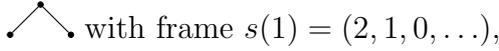
\begin{tikzpicture}[scale=0.4]
\tikzstyle{every node}=[draw,circle,fill=black,minimum size=2pt,inner sep=0pt,radius=1pt]
      \node (A) at (0,0) {};
       \node (B) at (1,1)  {} ;
      \node (C) at (2,0) {};
       \draw (A)--(B);
      \draw (B)--(C);
       \end{tikzpicture}
\vspace{.3cm}
with frame $s(1)=(2,1,0,\ldots)$,

lifted gives:

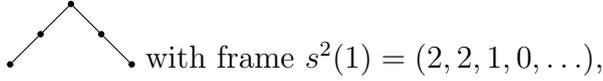
\begin{tikzpicture}[scale=0.4]
\tikzstyle{every node}=[draw,circle,fill=black,minimum size=2pt,inner sep=0pt,radius=1pt]
      \node (A) at (0,0) {};
       \node (B) at (1,1)  {} ;
      \node (C) at (2,2) {};
      \node (D) at (3,1) {};
      \node (E) at (4,0) {};
       \draw (A)--(B);
      \draw (B)--(C);
      \draw (C)--(D);
      \draw (D)--(E);
       \end{tikzpicture}
\vspace{.3cm}
with frame $s^2(1)=(2,2,1,0,\ldots)$,

extended  gives:

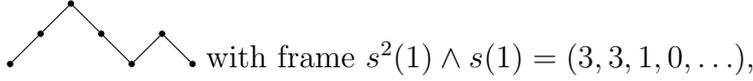
\begin{tikzpicture}[scale=0.4]
\tikzstyle{every node}=[draw,circle,fill=black,minimum size=2pt,inner sep=0pt,radius=1pt]
      \node (A) at (0,0) {};
       \node (B) at (1,1)  {} ;
      \node (C) at (2,2) {};
      \node (D) at (3,1) {};
      \node (E) at (4,0) {};
      \node (F) at (5,1) {};
      \node (G) at (6,0) {};
       \draw (A)--(B);
      \draw (B)--(C);
      \draw (C)--(D);
      \draw (D)--(E);
      \draw (E)--(F);
      \draw (F)--(G);
       \end{tikzpicture}
\vspace{.3cm}
with frame
$s^2(1)\wedge s(1)=(3,3,1,0,\ldots)$,

lifted gives:

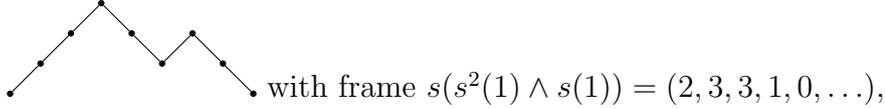
\begin{tikzpicture}[scale=0.4]
\tikzstyle{every node}=[draw,circle,fill=black,minimum size=2pt,inner sep=0pt,radius=1pt]
      \node (A) at (0,0) {};
       \node (B) at (1,1)  {} ;
      \node (C) at (2,2) {};
      \node (D) at (3,3) {};
      \node (E) at (4,2) {};
      \node (F) at (5,1) {};
      \node (G) at (6,2) {};
      \node (H) at (7,1) {};
      \node (I) at (8,0) {};
       \draw (A)--(B);
      \draw (B)--(C);
      \draw (C)--(D);
      \draw (D)--(E);
      \draw (E)--(F);
      \draw (F)--(G);
      \draw (G)--(H);
      \draw (H)--(I);
       \end{tikzpicture}
\vspace{.3cm}
with frame
$s(s^2(1)\wedge s(1))=(2,3,3,1,0,\ldots)$,

extended  gives:

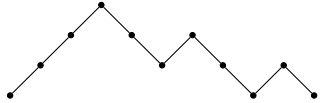
\begin{tikzpicture}[scale=0.4]
\tikzstyle{every node}=[draw,circle,fill=black,minimum size=2pt,inner sep=0pt,radius=1pt]
      \node (A) at (0,0) {};
       \node (B) at (1,1)  {} ;
      \node (C) at (2,2) {};
      \node (D) at (3,3) {};
      \node (E) at (4,2) {};
      \node (F) at (5,1) {};
      \node (G) at (6,2) {};
      \node (H) at (7,1) {};
      \node (I) at (8,0) {};
      \node (J) at (9,1) {};
      \node (K) at (10,0) {};
       \draw (A)--(B);
      \draw (B)--(C);
      \draw (C)--(D);
      \draw (D)--(E);
      \draw (E)--(F);
      \draw (F)--(G);
      \draw (G)--(H);
      \draw (H)--(I);
      \draw (I)--(J);
      \draw (J)--(K);
       \end{tikzpicture}
\vspace{.3cm}

with frame
$s(s^2(1)\wedge s(1)) \wedge s(1)=(3,4,3,1,0,\ldots)$
\end{ex}

Applying this procedure to any frame $(i_0, i_1,...,i_f,0,...)$, we reach a particular path belonging to this frame. This path is called the  {\it  canonical representative } of the frame $(i_0, i_1,...,i_f,0,...)$. In the preceding example the canonical representative is therefore  $s(s^2(\mathbf 0)\wedge s(\mathbf 0)) \wedge s(\mathbf 0)$.

The procedure is such that any time in a canonical representative there is a product of the form  $s^{j_1}(\mathbf x_1) \wedge s^{j_2}(\mathbf x_2) \wedge \ldots\wedge s^{j_t}(\mathbf x_t)$, then $s^{j_2}(\mathbf x_2)=\cdots=s^{j_t}(\mathbf x_t) =s(\mathbf 0)$.

In other words, a canonical path is of the form
$$s^{j_1}\!(s^{j_2}(...(\!s^{j_{t-1}}(\!s^{j_t}(\mathbf 0)\wedge \underbrace{s(\mathbf 0)\!\wedge...\wedge s(\mathbf 0)}_{k_t})\wedge \underbrace{s(\mathbf 0)\wedge...\wedge s(\mathbf 0)}_{k_{t-1}})...)\wedge \underbrace{s(\mathbf 0)\wedge...\wedge s(\mathbf 0)}_{k_1}.$$

Another way to describe the canonical path is to observe that in the corresponding sequence  of  U and D, after any sequence of one or more  D there is at most one  U.

\begin{ex} Assume  the frame $(3,6,6,3,1,0,...)$ is given.
By the preceding remarks to reach level 4 one must necessarily begin with a sequence of 4  U. Since there is a single foot at level 4 we must go down with at least  2 D. If we go down with  3 D, we could not go back to level  3, where there should be  3 feet. Hence we go down exactly two steps D and go back up with one  U. So far the sequence is  UUUUDDU. Since we do not need to go back to level 3 any longer, we go down with at least 2 more  D's. Again in this case, considering that at level 2 we have 6 feet, we must go down exactly with  2 D, then go up with one  U, then go down with one  D and climb back up with one  U two more times. We thus obtain  UUUUDDUDDUDUDU so far. Having exhausted the level 2 feet, we must go down with  2 D to level 0. Having obtained so far only 5 feet at level 1, we must still go up with  U and the concluding with a  D. The sequence is therefore UUUUDDUDDUDUDUDDUD.
\end{ex}

The following  property is often  useful in computations.
\m\noi
\begin{prop}\label{identita}
 $s(p(x)\wedge q(x))\wedge s(1) = s(p(x))\wedge s(q(x))$.
\end{prop}
\s\noi
\begin{proof} Use  Proposition \ref{polinomi}: $s(p(x)\wedge q(x))\wedge s(1) = s(p(x)+q(x)-1)\wedge (2+x) = (2+x(p(x)+q(x)-1)) \wedge (2+x) = 2+x(p(x)+q(x)-1) +(2+x) -1 = 2+xp(x)+2+xq(x)-1 = s(p(x))\wedge s(q(x))$.
\end{proof}
\m
From this, it immediately follows:

\m\noi
\begin{cor}\label{stessoschema} Given two Dyck paths $\mathbf x_1$, $\mathbf x_2$, the two paths $s(\mathbf x_1)\wedge s(\mathbf x_2)$, $s(\mathbf x_1\wedge \mathbf x_2)\wedge s(\mathbf 0)$ have the same frame.
\end{cor}
\m\noi
\begin{thm}\label{canonical}
     It is possible to obtain  the canonical representative of a frame $(i_0,i_1,...,i_f,0,...)$ in a finite number of steps starting from any representative path $\mathbf x$ using the commutativity property of the gluing operation and  Corollary \ref{stessoschema}.
\end{thm}
\s\noi
\begin{proof} If   $\mathbf x$ is a representative of the frame, suppose that $\mathbf x$ contains a  product $s(\mathbf 0)\wedge s^i(\mathbf y)$, then this can be transformed into   $s^i(\mathbf y)\wedge s(\mathbf 0)$.
If it contains a product $s^i(\mathbf y)\wedge s^j(\mathbf z)$, with $s^i(\mathbf y)\ne s(\mathbf 0), s^j(\mathbf z)\ne s(\mathbf 0)$, this can be transformed into  $s(s^{i-1}(\mathbf y)\wedge s^{j-1}(\mathbf z))\wedge s(\mathbf 0)$. After a finite number of steps of these two types, one gets to a canonical path when none of these two steps is any longer possible.
\end{proof}
\m

Using Theorem \ref{canonical}, we are now able to count the number of paths in a given frame.

\begin{ex} Consider the frame  $(3,4,3,1,0,\ldots)$ as before. We already saw that its canonical representative is $\mathbf u=s(s^2(\mathbf 0)\wedge s(\mathbf 0)) \wedge s(\mathbf 0)$. From this canonical representative, using commutativity of the wedge operation, we have a total of 4 paths, including $\mathbf u$, precisely:

\noi
$s(s^2(\mathbf 0)\wedge s(\mathbf 0)) \wedge s(\mathbf 0)$,\, $s(s(\mathbf 0)\wedge s^2(\mathbf 0)) \wedge s(\mathbf 0)$,\, $s(\mathbf 0)\wedge s(s^2(\mathbf 0)\wedge s(\mathbf 0))$,\, $s(\mathbf 0)\wedge s(s(\mathbf 0)\wedge s^2(\mathbf 0))$:
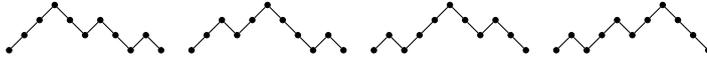
\vspace{1cm}

\begin{tikzpicture}[scale=0.2]
     \tikzstyle{every node}=[draw,circle,fill=black,minimum size=2pt,inner sep=0pt,radius=1pt]
       \node (A) at (0,0) {};
       \node (B) at (1,1)  {} ;
      \node (C) at (2,2) {};
      \node (D) at (3,3) {};
       \node (E) at (4,2) {};
       \node (F) at (5,1) {};
       \node (G) at (6,2) {};
       \node (H) at (7,1) {};
       \node (I) at (8,0) {};
       \node (J) at (9,1) {};
       \node (K) at (10,0) {};
       \draw (A)--(B);
      \draw (B)--(C);
      \draw (C)--(D);
      \draw (D)--(E);
      \draw (E)--(F);
      \draw (F)--(G);
      \draw (G)--(H);
      \draw (H)--(I);
      \draw (I)--(J);
      \draw (J)--(K);

      \node (A1) at (12,0) {};
       \node (B1) at (13,1)  {} ;
      \node (C1) at (14,2) {};
      \node (D1) at (15,1) {};
       \node (E1) at (16,2) {};
       \node (F1) at (17,3) {};
       \node (G1) at (18,2) {};
       \node (H1) at (19,1) {};
       \node (I1) at (20,0) {};
       \node (J1) at (21,1) {};
       \node (K1) at (22,0) {};
       \draw (A1)--(B1);
      \draw (B1)--(C1);
      \draw (C1)--(D1);
      \draw (D1)--(E1);
      \draw (E1)--(F1);
      \draw (F1)--(G1);
      \draw (G1)--(H1);
      \draw (H1)--(I1);
      \draw (I1)--(J1);
      \draw (J1)--(K1);

      \node (A2) at (24,0) {};
       \node (B2) at (25,1)  {} ;
      \node (C2) at (26,0) {};
      \node (D2) at (27,1) {};
       \node (E2) at (28,2) {};
       \node (F2) at (29,3) {};
       \node (G2) at (30,2) {};
       \node (H2) at (31,1) {};
       \node (I2) at (32,2) {};
       \node (J2) at (33,1) {};
       \node (K2) at (34,0) {};
       \draw (A2)--(B2);
      \draw (B2)--(C2);
      \draw (C2)--(D2);
      \draw (D2)--(E2);
      \draw (E2)--(F2);
      \draw (F2)--(G2);
      \draw (G2)--(H2);
      \draw (H2)--(I2);
      \draw (I2)--(J2);
      \draw (J2)--(K2);

      \node (A3) at (36,0) {};
       \node (B3) at (37,1)  {} ;
      \node (C3) at (38,0) {};
      \node (D3) at (39,1) {};
       \node (E3) at (40,2) {};
       \node (F3) at (41,1) {};
       \node (G3) at (42,2) {};
       \node (H3) at (43,3) {};
       \node (I3) at (44,2) {};
       \node (J3) at (45,1) {};
       \node (K3) at (46,0) {};
       \draw (A3)--(B3);
      \draw (B3)--(C3);
      \draw (C3)--(D3);
      \draw (D3)--(E3);
      \draw (E3)--(F3);
      \draw (F3)--(G3);
      \draw (G3)--(H3);
      \draw (H3)--(I3);
      \draw (I3)--(J3);
      \draw (J3)--(K3);
       \end{tikzpicture}

\m
Using  Corollary \ref{stessoschema} we have also  $\mathbf v= s^3(\mathbf 0)\wedge s^2(\mathbf 0)$, which gives rise, by the  commutativity of $\wedge$, to a total of  2 paths, including $\mathbf v$:

\noi
$s^3(\mathbf 0)\wedge s^2(\mathbf 0)$, \quad $s^2(\mathbf 0)\wedge s^3(\mathbf 0)$.

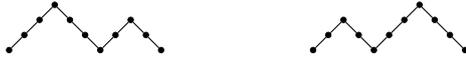
\vspace{1cm}
\begin{tikzpicture}[scale=0.2]
\tikzstyle{every node}=[draw,circle,fill=black,minimum size=2pt,inner sep=0pt,radius=1pt]
      \node (A) at (13,0) {};
       \node (B) at (14,1)  {} ;
      \node (C) at (15,2) {};
      \node (D) at (16,3) {};
       \node (E) at (17,2) {};
       \node (F) at (18,1) {};
       \node (G) at (19,0) {};
       \node (H) at (20,1) {};
       \node (I) at (21,2) {};
       \node (J) at (22,1) {};
       \node (K) at (23,0) {};
       \draw (A)--(B);
      \draw (B)--(C);
      \draw (C)--(D);
      \draw (D)--(E);
      \draw (E)--(F);
      \draw (F)--(G);
      \draw (G)--(H);
      \draw (H)--(I);
      \draw (I)--(J);
      \draw (J)--(K);

      \node (A3) at (33,0) {};
       \node (B3) at (34,1)  {} ;
      \node (C3) at (35,2) {};
      \node (D3) at (36,1) {};
       \node (E3) at (37,0) {};
       \node (F3) at (38,1) {};
       \node (G3) at (39,2) {};
       \node (H3) at (40,3) {};
       \node (I3) at (41,2) {};
       \node (J3) at (42,1) {};
       \node (K3) at (43,0) {};
       \draw (A3)--(B3);
      \draw (B3)--(C3);
      \draw (C3)--(D3);
      \draw (D3)--(E3);
      \draw (E3)--(F3);
      \draw (F3)--(G3);
      \draw (G3)--(H3);
      \draw (H3)--(I3);
      \draw (I3)--(J3);
      \draw (J3)--(K3);
       \end{tikzpicture}

\m

\noi We have therefore a total of 6 paths belonging to the frame $(3,4,3,1,0,\ldots)$.
\end{ex}

We wish to compute an explicit formula for the number  $p^{2n}_{\mathbf{I}}$ of paths having frame  $\mathbf{I}=(i_0,i_1,\ldots)$ with length $2n$.

\begin{defn}
  We define a {\it right frame} to be any frame of the form $$(2,i_1,i_2,i_3,\ldots)$$ and a {\it left frame} any other frame.
\end{defn}

\begin{defn}
  Given a frame $(i_0,i_1,i_2,\ldots)$, with $i_0\geq 2$, we define its {\it right progenitor} to be the frame $(2,i_1-i_0+2,i_2,\ldots)$. If the frame has $i_0=i_1=\ldots =i_{t-1}=2$ and $i_t\ne 2$   we define its {\it left progenitor} to be $(i_t,i_{t+1},\ldots)$. If $i_0\ne 2$ its left progenitor is itself.
\end{defn}

\begin{rem}
  Every frame has a right and left progenitor except for the null frame  which has only itself as a left progenitor and no right progenitor.
\end{rem}
We can prove two propositions describing how the cardinality of a frame $\mathbf{I}$ is related to the one of its progenitors.
\begin{prop}
  The cardinality of a frame $\mathbf{I}$ is the same as that of its left progenitor.
\end{prop}
\begin{proof}
  It is easy to see that there is a bijection between the set of paths realizing the frame $\mathbf{I}$ and those realizing the frame $s(\mathbf{I})$. Indeed, if $\mathbf p$ is a path of the frame $\mathbf{I}$ then $s(\mathbf p)$ is a path of the frame $s(\mathbf{I})$, moreover $s$ is an invertible operation. By repeatedly applying this argument we reach its left progenitor.
\end{proof}
Obviously,
\begin{lem}\label{sega}
  The cardinality of the frame $(u+1,u,0,\ldots)$, $\forall u \in \mathbb N$, is $1$.
\end{lem}

\begin{prop}\label{card1}
  Given a frame of the form $\mathbf X=(2+n,a,b\ldots)$ with $n\geq 0$, let $\mathbf Y=(2,a-n,b,\ldots)$ be its right  progenitor and  $k$ the cardinality of $\mathbf Y$. Then the cardinality of $\mathbf X$ is
  $$
  k\binom{a-1}{a-n-1}.
  $$
\end{prop}

  \begin{proof}
  If $b=0$ Theorem \ref{immediate} implies that the frame is of the form  $(u+1,u,0,\ldots)$.  The right progenitor is $(2,1,0,\ldots)$ with cardinality $k=1$ and the formula holds by Lemma \ref{sega}.
  If $b>0$ then  $a\geq 2+n$, by Theorem \ref{CNES}.
There is a bijection between the set of paths realizing the frame $\mathbf X$ and the weak compositions of $n$ into $a-n$ parts. Hence the formula.
In fact, there are $a-n$  positions  in $\mathbf Y$, ($a-n-2$ feet at level 1 of the path plus the two end-points at level 0) in which one can distribute the $n$ paths $s(\mathbf 0)$,  using Proposition \ref{identita} or the commutativity property. This means that each of the $n$ ``hats'' must be located in $a-n$ positions.

In other words, any path in $\mathbf Y$ gives rise to a path in the frame $\mathbf X$ by inserting $n$ pairs $01$. These pairs may be distributed
 either at the beginning, in the order $01$, or at the end  as $10$, or in the interior  for every possible pair $12$ as $1012$.
We are left to prove that if we take two different paths $\mathbf p$ and $\mathbf p'$ in the frame $\mathbf Y$  they give rise to different paths in $\mathbf X$.
Notice that there are no 0's in the middle of the associated sequences to $\mathbf p$ and $\mathbf p'$.
In the first position where they differ, one has the pair $a, a+1$ and the other has $a,a-1$ where $a\geq 2$. If $a>2$,  a pair $01$ or $10$ cannot be inserted after $a$ so the paths remain different. If $a=2$ then we have the sequence $\ldots abc21\ldots$ and $\ldots abc23\ldots$. The insertion of pairs $1,0$ or $0,1$ cannot turn these two sequences into equal ones.


  \end{proof}

\begin{thm}
 Given the frame $\mathbf{I}=(i_0,i_1,\ldots,i_f,0\ldots)$, setting:

$j_1=i_0-2,$

$j_2=i_1-i_0,$

$j_3=i_2-i_1+i_0-2,$

$j_4=i_3-i_2+i_1-i_0,$

$\ldots$

the  cardinality of $\mathbf{I}$ is

$${\binom{i_1-1}{i_1-j_1-1}}{\binom{i_2-1}{i_2-j_2-1}} \cdots {\binom{i_f-1}{i_f-j_f-1}}.$$

\end{thm}
\begin{proof}

Step $A_1$: subtract $j_1$ from the first two  elements of $\mathbf{I}$. We get the right  progenitor of $\mathbf{I}$:

$$\mathbf x_1 = b^{j_1}(\mathbf{I})=(2,i_1-j_1,i_2,...) = (2,i_1-i_0+2,i_2,...) = (2,2+j_2,i_2,...).$$

\m\noi
Step $B_1$: remove the initial  2 obtaining:
$$\mathbf y_1 =r(\mathbf x_1) =(2+j_2,i_2,...).$$
Notice that if $j_2>0$ then $\mathbf y_1$ is the left progenitor of $\mathbf x_1$.

By the previous propositions we have:

\begin{equation*}
  \begin{split}
& |\mathbf x_1|= |\mathbf y_1|\\
&|\mathbf{I}|= \binom{i_1-1}{i_1-j_1-1}|\mathbf x_1| = \binom{i_1-1}{i_1-j_1-1}|\mathbf y_1|.
  \end{split}
\end{equation*}
The second  identity follows from Proposition \ref{card1} if $j_1>0$ while it is trivial if $j_1=0$ (in this case $\mathbf x_1=\mathbf{I}$).

\m\noi
For $k> 1$, step $A_{k}$: subtract $j_{k}$ to the first two elements of $\mathbf y_{k-1}$. We get the right progenitor of $\mathbf y_{k-1}$:

$$\mathbf x_k = b^{j_k}(\mathbf y_{k-1})=(2,i_k-j_k,i_{k+1},...)  = (2,2+j_{k+1},i_{k+1},...).$$
Step $B_{k}$: Remove the initial 2 to obtain:
$$\mathbf y_k =r(\mathbf x_k)=(2+j_{k+1},i_{k+1},...).$$
Notice that if $j_{k+1}>0$ then $\mathbf y_k$ is the left progenitor of $\mathbf x_k$.
By the previous propositions we have:
\begin{equation*}
  \begin{split}
&|\mathbf x_k|= |\mathbf y_k|\\
&|\mathbf y_{k-1}|= \binom{i_k-1}{i_k-j_k-1}|\mathbf x_k| = \binom{i_k-1}{i_k-j_k-1}|\mathbf y_k|.
  \end{split}
\end{equation*}
The  second identity follows from  Proposition \ref{card1} if  $j_k>0$ while it is trivial if $j_k=0$ (in such case $\mathbf x_k=\mathbf y_{k-1}$).

We may deduce:

 $$|\mathbf{I}|=  \binom{i_1-1}{i_1-j_1-1} \ldots\binom{i_k-1}{i_k-j_k-1}|\mathbf y_k|.$$
Finally,
Step $A_{f-1}$: subtract $j_{f-1}$ to the first two elements of $\mathbf y_{f-2}$. We get the right progenitor of $\mathbf y_{f-2}$:
$$\mathbf x_{f-1} = b^{j_{f-1}}(\mathbf y_{f-2})=(2,i_{f-1}-j_{f-1},i_f,0,...)  = (2,2+j_{f},i_f,0,...).$$

\noi
Step $B_{f-1}$: Remove the initial 2 to obtain:
$$\mathbf y_{f-1} =r(\mathbf x_{f-1})=(2+j_{f},i_f,0,,...),$$
with
 $|\mathbf x_{f-1}|= |\mathbf y_{f-1}|$ and
$$|\mathbf y_{f-2}|= \binom{i_{f-1}-1}{i_{f-1}-j_{f-1}-1}|\mathbf x_{f-1}| = \binom{i_{f-1}-1}{i_{f-1}-j_{f-1}-1}|\mathbf y_{f-1}|,$$
hence
$$|\mathbf{I}|=  \binom{i_1-1}{i_1-j_1-1} \ldots\binom{i_{f-1}-1}{i_{f-1}-j_{f-1}-1}|\mathbf y_{f-1}|.$$
By Proposition \ref{CNES}, $\mathbf{y}_{f-1}$ is necessarily of the form $(u+1,u,0,...)$, with right  progenitor $\mathbf{x}_f=(2,1,0,...)$, which has  cardinality 1. By Proposition \ref{card1}, we have:

\noi
$|\mathbf{y}_{f-1}| = |\mathbf{x}_f|\binom{u-1}{0}=1$. Finally, since  Theorem \ref{CNES} implies $i_f-j_f-1=0$, we have $\binom{i_f-1}{i_f-j_f-1}=\binom{i_f-1}{0}=1$, hence the conclusion.

\end{proof}
\begin{ex}
  The cardinality of the frame $\mathbf{I}=(5,8,7,3)$ is
  $$
  \binom{7}{4}\binom{6}{3}\binom{2}{0}=700
  $$
\end{ex}

\section{Colored  Dyck paths}
For any fixed frame  $(i_0,\ldots, i_f,\,\ldots)$ with length $2n$, denote by  $v_k$ the number of U steps joining the levels  $k$, $k+1$ ($k=0,...,f-1$). Such number is clearly equal to the number of the D steps joining the same two levels.

\begin{thm}
In the above notation:
 $$v_k= i_k-i_{k-1}+...+(-1)^k i_0+ (-1)^{k+1},\quad k=0,...,f-1.$$

\end{thm}
\begin{proof} Since at level  0 from the first node we have a  U step, and we also have a D step in the final node, while in all the others we certainly have a U and a D, we have a total of $2(i_0-1)$ steps joining level 0 and level 1. Since there are as many U as D we must have:
$v_0=i_0-1$.
Assume we proved the formula up to  $v_{k-1}$, we prove it for  $v_k$. From the  $i_k$ nodes at level  $k$ we have  $2i_k$ steps, half of which are  U and half are D. Of these,  $2v_{k-1}$ come from the lower level. It follows that $v_k= i_k -v_{k-1} =i_k - (i_{k-1}-...+(-1)^{k -1}i_0+ (-1)^{k})$, hence the result.
\end{proof}
\m
Notice that the number $v_k$ depends only on the given frame and not on the particular path in that frame.

Denote by  $I_{2n}$ the set of frames  with length $2n$. Suppose that we can color with $u_k$ colors the  U steps joining level $k$ and level $k+1$ and with  $d_k$ colors the D steps joining the same levels ($k=0,...,n-1$). The number of colored paths of a given frame  $\mathbf{I}=(i_0,i_1,\ldots)$ with length $2n$, is clearly
$p^{2n}_{\mathbf{I}}u_0^{v_0}...u_{n-1}^{v_{n-1}}d_0^{v_0}...d_{n-1}^{v_{n-1}}$, recalling that $p^{2n}_{\mathbf{I}}$ denotes the number of elements in the frame  $\mathbf{I}=(i_0,i_1,\ldots)$ with length $2n$. We immediately have

\m\noi
\begin{thm}
 The number of Dyck paths with length  $2n$, that can be colored with  $u_k$ and $d_k$ colors from level  $k$ to level $k+1$ ($k=0,...,n-1$) is
$$ \sum_{\mathbf{I} \in I_{2n}} p^{2n}_{\mathbf{I}}u_0^{v_0}...u_{n-1}^{v_{n-1}}d_0^{v_0}...d_{n-1}^{v_{n-1}}.$$
\end{thm}
\m
Notice that, the number  $v_k$ depends also on the frame $\mathbf I$. However, to explicitly indicate such relation in the notation would make for a quite cumbersome symbol such as $v_k(p^{2n}_{\mathbf{I}})$.

\section{ Colored Motzkin paths}

\begin{prop}
  If $m^{(n)}$ is the number of Motzkin paths with length  $n$, and we set $\nu={\left[ {\frac{n}{2}}\right]}$, then
  \begin{equation}\label{motzkin1}
m^{(n)}= \sum_{j=0}^{\nu}\sum_{\mathbf{I} \in  I_{2j}} p^{2j}_{\mathbf{I}} {\binom{n}{n-2j}}.
\end{equation}

\end{prop}
\begin{proof}
Analogously to what was observed in the case of Theorem \ref{thkmotzkin}, a Motzkin path with length $n$ may be obtained from a Dyck path  $\mathbf p$ with length  $2j$, $0\le 2j\le n$, by adding  $n-2j$ horizontal steps at the feet at various levels of $\mathbf p$. If $\mathbf p$ has frame $\mathbf{I}=(i_0,\ldots,)$, we may distribute the $n-2j$ horizontal steps at each of the $2j+1$ nodes of $\mathbf p$. So this can be done in as many ways as the  possibility to express  $n-2j$ as a sum of  $2j+1$  integers between  0 and $n-2j$, that is,  in $\binom{n-2j+(2j+1)-1}{n-2j}=\binom{n}{n-2j}$ ways.  The statement follows.
\end{proof}

 The $n-2j$ horizontal steps can be distributed as follows: $k_0$ steps at each of the $i_0$ nodes of level 0, $k_1$ steps  at  the $i_1$ nodes of level 1, and so on. The number of such possible arrangements are counted by weak compositions of suitable integers.  If at each level  $r$ one uses  $h_{r}$ colors on the horizontal steps and denotes by $\mathcal H=(h_0,\ldots, h_{\nu})$, the formula becomes
 \begin{prop}
$$m_{\mathcal H}^{(n)} \!= \!\sum_{j=0}^{\nu}\sum_{\mathbf{I} \in  I_{2j}} p^{2j}_{\mathbf{I}} \sum_{k_0+\cdots +k_\nu=n-2j}\!{\binom{k_0+\!i_{0}\!-\!1}{k_0}}\!\cdots \!{\binom{k_\nu +\!i_{\nu}\!-\!1}{k_\nu}}h^{k_0}_{0}\!\cdots \!h^{k_{\nu}}_{\nu}.$$
\end{prop}

It may happen that in the summation  corresponding to a   frame, the frame is too ``low'' to allow adding horizontal steps, for example at level $\nu$. In this case the corresponding binomial coefficient  ${\binom{k_\nu+\!i_{\nu}\!-\!1}{k_\nu}}$  has  $i_{\nu}=0$ and is therefore zero.

If we add to this the possibility of coloring the U and D steps with  $u_k$ e $d_k$ colors from level $k$ to level $k+1$ ($k=0,...,\nu-1$),
and we denote by $\mathcal{U}=(u_0, \ldots, u_{\nu -1}),\mathcal{D}=(d_0, \ldots, d_{\nu -1})$,
one gets

\begin{thm} Denoting by $m_{\mathcal{H},\mathcal{U},\mathcal{D}}^{(n)}$ the number of Motzkin paths colored with colors determined by $\mathcal{H},\mathcal{U},\mathcal{D}$, we have that
$m_{\mathcal{H},\mathcal{U},\mathcal{D}}^{(n)}$ is equal to
\begin{equation}\label{motzkin2}
   \sum_{\substack{j=0,\ldots,\nu\\\mathbf{I} \in  I_{2j}}}\!\!\! p^{2j}_{\mathbf{I}} u_0^{v_0}...u_{j-1}^{v_{j-1}}d_0^{v_0}...d_{j-1}^{v_{j-1}}\!\!\!\sum_{\substack{k_0+\cdots +k_{\nu}\\=n-2j}}\!{\binom{k_0+\!i_{0}\!-\!1}{k_0}}\!\cdots \!{\binom{k_{\nu}+\!i_{{\nu}}\!-\!1}{k_{\nu}}}h^{k_0}_{0}\!\cdots \!h^{k_{\nu}}_{{\nu}}.
\end{equation}

\end{thm}

Notice that such formula makes sense  if we set $h_k=0$ for all those levels $k$ where there are no horizontal steps provided we  attribute value  1 to the expression $h_j^{k_j}$, if $h_j=0$ and $k_j=0$.

\begin{rem}
Comparing formula (\ref{motzkin1}) with (\ref{motzkin2}) written in the case of all $h_k=u_k=d_k=1$ yields the following identity for binomial coefficients, where we made obvious changes of symbols:
$${\binom{m+{i_0}+\cdots +i_\nu -1}{m}}  = \sum_{k_0+\cdots +k_\nu=m}\!{\binom{k_0+\!{i_0}\!-\!1}{k_0}}\!\cdots \!{\binom{k_\nu+\!{i_\nu}\!-\!1}{k_\nu}}.$$
\end{rem}

\end{document}